\documentclass[12pt,reqno]{amsart} 

\usepackage{mathrsfs,amsfonts,amsthm,amsmath,amssymb,thmtools} 
\usepackage[dvipsnames]{xcolor} 

\usepackage{enumitem}                
\usepackage{lineno}                 

\usepackage[numbers,sort&compress]{natbib}  

\usepackage{graphics}
\usepackage{graphicx}   
\graphicspath{{images/}} 
\usepackage{subfig}     
\usepackage{rotating}

\usepackage[hyperindex, pdfencoding = auto, psdextra, bookmarksdepth = 4]{hyperref}  
\hypersetup{
    bookmarksopen      = true,
    bookmarksopenlevel = 1,
    bookmarksnumbered  = true,
    pdftitle    = {},
    pdfauthor   = {},
    linktoc     = page,
    anchorcolor = green,
    colorlinks  = true,        
    linkcolor   = black,
    citecolor   = blue,
    filecolor   = blue,
    urlcolor    = magenta,
}

\usepackage[capitalise, noabbrev, nameinlink]{cleveref}    

\theoremstyle{plain}

\newtheorem{theorem}{Theorem}[section]
\newtheorem{lemma}[theorem]{Lemma}

\newtheorem{proposition}[theorem]{Proposition}

\theoremstyle{definition}   

\makeatletter
    \def\subsection{\@startsection{subsection}{2}%
    \z@{.5\linespacing\@plus.7\linespacing}{.3\linespacing}%
    {\normalfont\bfseries}}
\makeatother 

\makeatletter                                         
    \newcommand{\LeftEqNo}{\let\veqno\@@leqno}        
\makeatother

\numberwithin{equation}{section}                      

\begin{document}

\

\vspace{-2cm}

\title{Single commutators in $C^*$-algebras}

\author{Lu Cui}
\address{Lu Cui, College of Science, China University of Petroleum (East China), Qingdao, 266580, China}
\email{nilu@upc.edu.cn}

\author{Minghui Ma}
\address{Minghui Ma, School of Mathematical Sciences, Dalian University of Technology, Dalian, 116024, China}
\email{minghuima@dlut.edu.cn}

\keywords{Single commutator, simple $C^*$-algebra, real rank zero, tracial state}

\makeatletter
\@namedef{subjclassname@2020}{%
  \textup{2020} Mathematics Subject Classification}
\makeatother

\subjclass[2020]{Primary 46L05; Secondary 47B47}

\begin{abstract}
Let $\mathcal A$ be a simple unital $C^*$-algebra with real rank zero, nonempty tracial state space $\mathcal T(\mathcal A)$, and strict comparison of projections.
We prove that every element $T$ in $\mathcal A$ with $\tau(T)=0$ for all $\tau\in\mathcal T(\mathcal A)$ is a single commutator if one of the following conditions holds:
$(1)$ $\mathcal A$ has unique trace; $(2)$ $\mathcal A$ has stable rank one.
We also prove that every nonscalar element in a simple unital purely infinite $C^*$-algebra is a single commutator.
Applications include UHF algebras, irrational rotation algebras, and Cuntz algebras.
\end{abstract}

\maketitle

\section{Introduction}

A longstanding problem in operator theory and operator algebras is to characterize \emph{commutators} in an algebra $\mathcal A$, namely, elements of the form $[A,B]=AB-BA$ with $A,B\in\mathcal A$.
By an early theorem of Shoda \cite{Shoda-36}, a finite-dimensional complex matrix is a single commutator if and only if it has trace zero.
On a separable infinite-dimensional Hilbert space $\mathcal H$, Brown and Pearcy \cite{BP-65} proved that an operator in $\mathcal B(\mathcal H)$ is a single commutator if and only if it is not of the form $\lambda I+K$ with $\lambda\ne 0$ and $K$ compact.
Their work on type $\mathrm{III}$ factors \cite{BP-66} and the subsequent papers of Brown, Pearcy, and Topping \cite{BPT-68}, Halpern \cite{Halpern-69}, and Popa \cite{Popa-82} extended this line of investigation to properly infinite von Neumann algebras.
In type $\mathrm{II}_1$ von Neumann algebras, Fack and de la Harpe \cite{FH-80} showed that every operator with center-valued trace zero is a sum of at most ten commutators with norm estimates.
Goldstein and Paszkiewicz \cite{GP-92} subsequently proved that every self-adjoint operator with center-valued trace zero is a sum of at most four commutators.
Later, Marcoux \cite{Marcoux-06} proved that every operator with trace zero in a type $\mathrm{II}_1$ factor is a sum of at most two commutators and that every self-adjoint operator with trace zero is a sum of at most four \emph{self-commutators}, that is, commutators of the form $[A,A^*]$.
For related results, see also Pearcy and Topping \cite{PT-69} and Dykema and Skripka \cite{DS-12}.
More recently, Wen, Fang, and Yao \cite{WFY-24} settled the single-commutator problem for type $\mathrm{II}_1$ factors by showing that every operator with trace zero is a single commutator.
They also proved that every self-adjoint operator with trace zero is a single self-commutator.

The commutator problem has also been studied extensively in $C^*$-algebras.
Fack \cite[Theorem 3.1]{Fack-82} showed that, in a unital simple AF-algebra, every self-adjoint element that vanishes under every tracial state is a sum of at most seven self-commutators.
Thomsen \cite{Thomsen-93} extended this approach to certain inductive limits of homogeneous algebras.
Subsequently, Marcoux \cite[Corollary 4.9]{Marcoux-06} reduced the number to two commutators in several classes of $C^*$-algebras, including UHF algebras and irrational rotation algebras.
Kaftal, Ng, and Zhang \cite[Theorem 3.4]{KNZ-14} obtained the same conclusion for simple unital separable $C^*$-algebras with real rank zero, stable rank one, and strict comparison of projections.
For simple unital purely infinite $C^*$-algebras, Pop \cite[Theorem 1 and Remark 3]{Pop-02} proved that every element is a sum of two commutators.
For further developments and related results, see \cite{CP-79,DFWW-04,DK-18,Marcoux-10,Marcoux-22,Ng-12,Ng-14,Ng-15,Robert-15}.
The aim of the current paper is to solve the longstanding single-commutator problem in several classes of $C^*$-algebras.

Throughout, let $\mathcal A$ be a simple unital $C^*$-algebra with identity $I$ and $\mathcal T(\mathcal A)$ the tracial state space of $\mathcal A$.
The set of all commutators in $\mathcal A$ is denoted by
\begin{equation*}
	\mathfrak{c}(\mathcal A)=\{[A,B]\colon A,B\in\mathcal A\}.
\end{equation*}
A simple unital $C^*$-algebra $\mathcal A$ with nonempty $\mathcal T(\mathcal A)$ is said to have \emph{strict comparison of projections} if for any two projections $P$ and $Q$ in $\mathcal A$, we have
\begin{equation*}
	\tau(P)<\tau(Q)~\text{for all}~\tau\in \mathcal T(\mathcal A)\Longrightarrow P\precsim Q.
\end{equation*}
For more related definitions, the reader is referred to \cite{Blackadar-88,Blackadar-98,BP-91}.
We now present our main theorems as follows.

\begin{theorem}\label{thm trace}
Let $\mathcal A$ be a simple unital $C^*$-algebra with real rank zero, nonempty tracial state space $\mathcal T(\mathcal A)$, and strict comparison of projections.
Then
\begin{equation*}
	\bigcap_{\tau\in \mathcal T(\mathcal A)}\ker\tau=\mathfrak{c}(\mathcal A)
\end{equation*}
whenever one of the following conditions holds: $(1)$ $\mathcal A$ has unique trace; $(2)$ $\mathcal A$ has stable rank one.
\end{theorem}

\begin{theorem}\label{thm traceless}
Let $\mathcal A$ be a simple unital purely infinite $C^*$-algebra.
Then
\begin{equation*}
	(\mathcal A\setminus\mathbb{C}I)\cup\{0\}=\mathfrak{c}(\mathcal A).
\end{equation*}
\end{theorem}

As applications, we obtain a complete characterization of single commutators in UHF algebras \cite{Glimm-60}, irrational rotation algebras \cite{Rieffel-81}, and Cuntz algebras \cite{Cuntz-77}.
We also recover the result \cite{WFY-24} that every operator with trace zero in a type $\mathrm{II}_1$ factor is a single commutator.
Moreover, the preceding theorems provide a partial solution to the general commutator characterization problem in Marcoux's survey \cite[Problem~3]{Marcoux-22} and affirmative answers to two questions in Ng's paper \cite[pp.~58--59]{Ng-15}.
It remains a challenging problem to determine the single commutators in the Jiang--Su algebra \cite{JS-99} and reduced free group $C^*$-algebras \cite{Powers-75}.

\section{Proofs}

We begin with an elementary lemma describing a standard matrix operation.

\begin{lemma}\label{lem 12=V}
Let $\mathcal A$ be a unital $C^*$-algebra and let $\{P_j\}_{j=1}^n$ be mutually orthogonal projections in $\mathcal A$ with sum $I$, where $n\geqslant2$.
Suppose that $\{V_{1j}\}_{j=2}^n$ are partial isometries in $\mathcal A$ satisfying
\begin{equation*}
    V_{12}V_{12}^*=P_1,\qquad
    V_{1j}V_{1j}^*\leqslant P_1
    \quad(3\leqslant j\leqslant n),\qquad
    V_{1j}^*V_{1j}=P_j
    \quad(2\leqslant j\leqslant n).
\end{equation*}
If $T\in\mathcal A$ satisfies $P_1TP_2=V_{12}$, then $T$ is similar to an element $S\in\mathcal A$ such that
\begin{equation*}
    P_1SP_2=V_{12},
    \qquad
    P_jSP_j=0\quad(3\leqslant j\leqslant n).
\end{equation*}
\end{lemma}

\begin{proof}
Let $Q=I-P_1-P_2$ and
\begin{equation*}
	X=\sum_{j=3}^{n}V_{12}^*(V_{1j}-P_1TP_j)\in P_2\mathcal AQ.
\end{equation*}
Then $X^2=0$.
Let $A=(I-X)T(I+X)$.
It is straightforward to verify that $P_1AP_j=V_{1j}$ for every $2\leqslant j\leqslant n$.
Now let
\begin{equation*}
	Y=\sum_{j=3}^{n}(P_jAP_j)V_{1j}^*\in Q\mathcal AP_1.
\end{equation*}
Then $Y^2=0$.
Let $S=(I-Y)A(I+Y)$.
Then $S$ has the desired properties.
\end{proof}

In the next lemma, we show that every nonscalar element satisfies the required condition in \Cref{lem 12=V} up to similarity.

\begin{lemma}\label{lem S12}
Let $\mathcal A$ be a simple unital $C^*$-algebra with real rank zero.
Then for every nonscalar element $T$ in $\mathcal A$, there are nonzero projections $P_1, P_2$ with $P_1P_2=0$, a partial isometry $V_{12}$, and an element $S$ similar to $T$ such that
\begin{equation*}
	V_{12}V_{12}^*=P_1,\qquad V_{12}^*V_{12}=P_2,\qquad P_1SP_2=V_{12}.
\end{equation*}
\end{lemma}

\begin{proof}
Since $\mathcal A$ is simple and unital, $\mathcal A$ has trivial center.
Since $\mathcal A$ has real rank zero and $T$ is nonscalar, there exists a projection $P$ such that $PT\ne TP$.
Replacing $P$ by $I-P$ if necessary, we may assume that
\begin{equation*}
	A:=(I-P)TP\ne 0.
\end{equation*}
Since $P\mathcal AP$ has real rank zero, there exists a positive finite-spectrum element $H$ in $P\mathcal AP$ such that
\begin{equation*}
	\|A^*A-H\|<\frac{1}{4}\|A\|^2.
\end{equation*}
Let $\lambda=\|H\|>\frac{3}{4}\|A\|^2$ and let $P_2$ be the nonzero spectral projection of $H$ with respect to the singleton $\{\lambda\}$.
Then $P_2\leqslant P$ and
\begin{equation*}
	P_2A^*AP_2\geqslant\left(\lambda-\frac{1}{4}\|A\|^2\right)P_2\geqslant\frac{1}{2}\|A\|^2P_2.
\end{equation*}
Hence $P_2A^*AP_2$ is invertible in $P_2\mathcal AP_2$.
Let $B=(P_2A^*AP_2)^{-1/2}$ be an invertible element in $P_2\mathcal AP_2$.
We define $V_{12}=AP_2B$ and $P_1=V_{12}V_{12}^*$.
Then
\begin{equation*}
	V_{12}^*V_{12}=B(P_2A^*AP_2)B=P_2.
\end{equation*}
Hence $V_{12}$ is a partial isometry and $P_1$ is a nonzero projection.
Since $(I-P)V_{12}=V_{12}$, we have $P_1\leqslant I-P$ and hence $P_1P_2=0$.
Let $X=B+(I-P_2)$ and $S=X^{-1}TX$.
Then $P_1SP_2=V_{12}$.
This completes the proof.
\end{proof}

For later use, we record the following two lemmas from Marcoux's paper \cite[Lemmas~2.2 and~3.8]{Marcoux-06}.
The proof of \Cref{lem diagonal-commutator} uses Rosenblum's theorem \cite{Rosenblum-56}.

\begin{lemma}\label{lem diagonal-commutator}
Let $\mathcal{A}$ be a unital $C^*$-algebra and let $\{P_j\}_{j=1}^n$ be mutually orthogonal projections in $\mathcal{A}$ with sum $I$.
Suppose that $T\in\mathcal{A}$ and there are elements $A_j,B_j$ in $P_j\mathcal AP_j$ satisfying
\begin{equation*}
    P_jTP_j=[A_j,B_j]\qquad(1\leqslant j\leqslant n).
\end{equation*}
Then $T$ is a single commutator.
\end{lemma}

\begin{lemma}\label{lem trace-PAP}
Let $\mathcal A$ be a simple unital $C^*$-algebra and let $P$ be a nonzero projection in $\mathcal A$.
Then every tracial state $\tau_P$ on $P\mathcal AP$ is of the form
\begin{equation*}
	\tau_P(\cdot)=\frac{\tau(\cdot)}{\tau(P)}
\end{equation*}
for some $\tau\in \mathcal T(\mathcal A)$.
\end{lemma}

For completeness, we include a proof of the next lemma, which follows directly from \Cref{lem diagonal-commutator}.

\begin{lemma}\label{lem 2.3}
Let $\mathcal B$ be a unital $C^*$-algebra.
Suppose that $T$ is an element in $M_2(\mathcal B)$ of the form
$T=
\begin{pmatrix}
	T_{11} & I_{\mathcal B}\\
	T_{21} & T_{22}
\end{pmatrix}$,
and $T_{11}+T_{22}$ is the sum of two commutators in $\mathcal B$.
Then $T$ is a single commutator in $M_2(\mathcal B)$.
\end{lemma}

\begin{proof}
By assumption, $T_{11}+T_{22}=[A_1,B_1]+[A_2,B_2]$ for some $A_1,B_1,A_2,B_2\in\mathcal B$.
Let $W=
\begin{pmatrix}
	I_{\mathcal B} & 0\\
	X & I_{\mathcal B}
\end{pmatrix}$ and $S=W^{-1}TW$, where $X\in\mathcal B$.
Then
\begin{equation*}
	S=
	\begin{pmatrix}
		T_{11}+X & I_{\mathcal B}\\
		T_{21}+T_{22}X-XT_{11}-X^2 & T_{22}-X
	\end{pmatrix}.
\end{equation*}
By taking $X=[A_1,B_1]-T_{11}$, we have
\begin{equation*}
	S=
	\begin{pmatrix}
		[A_1,B_1] & I_{\mathcal B}\\
		S_{21} & [A_2,B_2]
	\end{pmatrix}.
\end{equation*}
It follows from \Cref{lem diagonal-commutator} that $S=[A,B]$ for some elements $A,B\in M_2(\mathcal B)$.
Thus, $T=WSW^{-1}=[WAW^{-1},WBW^{-1}]$.
This completes the proof.
\end{proof}

The next technical lemma is obtained by repeatedly applying Zhang's result \cite[Corollary 1.3]{Zhang-90-JOT}.

\begin{lemma}\label{lem P-decomposition-Q}
Let $\mathcal A$ be a simple unital $C^*$-algebra with real rank zero.
Then for any nonzero projection $Q$ and any projection $P$, there is a finite orthogonal decomposition
\begin{equation*}
	P=P_1+P_2+\cdots+P_n
\end{equation*}
such that $P_j\precsim Q$ for every $1\leqslant j\leqslant n$.
\end{lemma}

\begin{proof}
Since $\mathcal A$ is simple and $Q$ is nonzero, there exists $n\geqslant 1$ such that
\begin{equation*}
	P\otimes e_{11}\precsim Q\otimes I_n\in M_n(\mathcal A).
\end{equation*}
Thus, there exists a partial isometry $W$ in $M_n(\mathcal A)$ such that
\begin{equation*}
	W^*W=P\otimes e_{11},\qquad R:=WW^*\leqslant Q\otimes I_n.
\end{equation*}
Since $M_n(\mathcal A)$ has real rank zero, repeated application of \cite[Corollary 1.3]{Zhang-90-JOT} gives mutually orthogonal projections $\{R_j\}_{j=1}^n$ satisfying
\begin{equation*}
	R=R_1+R_2+\cdots+R_n,\qquad R_j\precsim Q\otimes e_{jj}\quad(1\leqslant j\leqslant n).
\end{equation*}
Since $W^*R_jW\leqslant W^*W=P\otimes e_{11}$, there exists a projection $P_j$ in $\mathcal A$ such that
\begin{equation*}
	P_j\otimes e_{11}=W^*R_jW.
\end{equation*}
Then $P=P_1+P_2+\cdots+P_n$ and
\begin{equation*}
	P_j\otimes e_{11}\sim R_j\precsim Q\otimes e_{jj}\sim Q\otimes e_{11}.
\end{equation*}
Thus, $P_j\precsim Q$ for every $1\leqslant j\leqslant n$.
This completes the proof.
\end{proof}

In the following proposition, we show how to reduce a two-commutator assumption to a single-commutator one.

\begin{proposition}\label{prop 2-to-1}
Let $\mathcal A$ be a simple unital $C^*$-algebra with real rank zero and nonempty $\mathcal T(\mathcal A)$.
Suppose that
\begin{equation*}
	\bigcap_{\tau_P\in \mathcal T(P\mathcal AP)}\ker\tau_P=\mathfrak{c}(P\mathcal AP)+\mathfrak{c}(P\mathcal AP)
\end{equation*}
for every nonzero projection $P$ in $\mathcal A$.
Then every element $T$ in $\mathcal A$ with $\tau(T)=0$ for all $\tau\in \mathcal T(\mathcal A)$ is a single commutator.
\end{proposition}

\begin{proof}
If $T$ is a scalar element, then $T=0$ because $\mathcal T(\mathcal A)$ is nonempty.
Thus, we may assume that $T$ is nonscalar.
By \Cref{lem 12=V}, \Cref{lem S12}, and \Cref{lem P-decomposition-Q}, there are projections $\{P_j\}_{j=1}^n$ with sum $I$, partial isometries $\{V_{1j}\}_{j=2}^n$, and an element $S$ similar to $T$ such that
\begin{equation*}
	V_{12}V_{12}^*=P_1,\qquad V_{1j}V_{1j}^*\leqslant P_1\quad(3\leqslant j\leqslant n),\qquad V_{1j}^*V_{1j}=P_j\quad(2\leqslant j\leqslant n),
\end{equation*}
and
\begin{equation*}
	P_1SP_2=V_{12},\qquad P_jSP_j=0\quad(3\leqslant j\leqslant n).
\end{equation*}
Let $D=P_1SP_1+V_{12}SV_{12}^*\in P_1\mathcal AP_1$.
By \Cref{lem trace-PAP}, for every $\tau_{P_1}\in \mathcal T(P_1\mathcal AP_1)$, we have
\begin{equation*}
	\tau_{P_1}(D)=\frac{\tau(P_1SP_1+V_{12}SV_{12}^*)}{\tau(P_1)}=\frac{\tau(S)}{\tau(P_1)}=\frac{\tau(T)}{\tau(P_1)}=0.
\end{equation*}
Hence $D=[A_1,B_1]+[A_2,B_2]$ for some $A_1,B_1,A_2,B_2\in P_1\mathcal AP_1$ by hypothesis.
By \Cref{lem 2.3}, $(P_1+P_2)S(P_1+P_2)$ is a single commutator in $(P_1+P_2)\mathcal A(P_1+P_2)$.
We complete the proof by \Cref{lem diagonal-commutator}.
\end{proof}

The proof of the following proposition is analogous to that of \Cref{prop 2-to-1} and is therefore omitted.

\begin{proposition}\label{prop 2-to-1-traceless}
Let $\mathcal A$ be a simple unital $C^*$-algebra with real rank zero.
Suppose that
\begin{equation*}
	P\mathcal AP=\mathfrak{c}(P\mathcal AP)+\mathfrak{c}(P\mathcal AP)
\end{equation*}
for every nonzero projection $P$ in $\mathcal A$.
Then every nonscalar element $T$ in $\mathcal A$ is a single commutator.
\end{proposition}

The following lemma allows us to apply \cite[Theorem~3.4]{KNZ-14} without assuming that $\mathcal A$ is separable in the proof of \Cref{thm trace}.

\begin{lemma}\label{lem separable-reduction}
Let $\mathcal A$ be a simple unital $C^*$-algebra with real rank zero, stable rank one, nonempty $\mathcal T(\mathcal A)$, and strict comparison of projections.
Suppose that $T\in\mathcal A$ satisfies $\tau(T)=0$ for every $\tau\in \mathcal T(\mathcal A)$.
Then there exists a separable simple unital $C^*$-subalgebra $\mathcal B\subseteq\mathcal A$ with real rank zero, stable rank one, nonempty $\mathcal T(\mathcal B)$, and strict comparison of projections such that $T\in\mathcal B$ and $\rho(T)=0$ for every $\rho\in \mathcal T(\mathcal B)$.
\end{lemma}

\begin{proof}
By \cite[Theorem~2.6.2]{FHLRTVW-21}, there exists a separable unital elementary subalgebra $\mathcal B\subseteq\mathcal A$ containing $T$.
Then $\mathcal B$ is simple and has real rank zero and stable rank one.
Since tracial states on $\mathcal A$ restrict to $\mathcal B$ and projection subequivalence is reflected by elementarity, $\mathcal T(\mathcal B)$ is nonempty and $\mathcal B$ has strict comparison of projections.
By the Hahn--Banach theorem, $T$ belongs to the norm-closed linear span of commutators in $\mathcal A$ (see also \cite{CP-79}).
By elementarity, the same holds in $\mathcal B$.
Hence $\rho(T)=0$ for every $\rho\in \mathcal T(\mathcal B)$.
\end{proof}

We can now prove the main theorems of this paper.

\begin{proof}[Proof of \Cref{thm trace}]
Let $P$ be a nonzero projection in $\mathcal A$.
It follows from \cite[Lemma 3.5]{Marcoux-06} that $P\mathcal AP$ is a simple unital $C^*$-algebra with real rank zero and nonempty tracial state space and has strict comparison of projections.

$(1)$ Suppose that $\mathcal A$ has unique trace.
By \Cref{lem trace-PAP}, $P\mathcal AP$ has a unique tracial state $\tau_P$.
It follows from \cite[Theorem 3.10]{Marcoux-06} that
\begin{equation*}
	\ker\tau_P=\mathfrak{c}(P\mathcal AP)+\mathfrak{c}(P\mathcal AP).
\end{equation*}

$(2)$ Suppose that $\mathcal A$ has stable rank one.
Then $P\mathcal AP$ has stable rank one.
By \Cref{lem separable-reduction} and \cite[Theorem~3.4]{KNZ-14}, we obtain
\begin{equation*}
    \bigcap_{\tau_P\in \mathcal T(P\mathcal AP)}\ker\tau_P
    =\mathfrak{c}(P\mathcal AP)+\mathfrak{c}(P\mathcal AP).
\end{equation*}
We finish the proof by \Cref{prop 2-to-1} in both cases.
\end{proof}

We now prove \Cref{thm traceless} by applying \Cref{prop 2-to-1-traceless} and the method adopted in the proof of \Cref{thm trace}.

\begin{proof}[Proof of \Cref{thm traceless}]
For every nonzero projection $P$ in $\mathcal A$, $P\mathcal AP$ is again a simple unital purely infinite $C^*$-algebra.
By \cite[Theorem 1 and Remark 3]{Pop-02}, we have 
\begin{equation*}
	P\mathcal AP=\mathfrak{c}(P\mathcal AP)+\mathfrak{c}(P\mathcal AP).
\end{equation*}
Since $\mathcal A$ has real rank zero by \cite{Zhang-90}, the conclusion follows from \Cref{prop 2-to-1-traceless}.
\end{proof}

\section*{Acknowledgments}
We are grateful to Professor Junsheng Fang for drawing our attention to this problem.
We thank Zhichao Liu for pointing out that the proof in the earlier version, given for UHF algebras, applies to a more general setting, which we adopt in the current version.

\end{document}